\documentclass[11pt]{amsart}
\usepackage{preamble}
\usepackage{mathtools}
\usepackage{tikz-cd}
\usetikzlibrary{matrix,arrows,decorations.pathmorphing}

\usepackage{hyperref}
\usepackage[colorinlistoftodos]{todonotes} % for comments

\newcommand{\GG}{\mathbb{G}}
\newcommand{\OO}{\mathbb{O}}
\newcommand{\XX}{\mathbb{X}}
\newcommand{\GC}{\mathcal{GC}}

\DeclareMathOperator{\gr}{gr}
\DeclareMathOperator{\Rect}{Rect}
\DeclareMathOperator{\Pent}{Pent}
\DeclareMathOperator{\Hex}{Hex}
\DeclareMathOperator{\CF}{CF}
\DeclareMathOperator{\CFK}{CFK}

\title{
    Crossing change maps in filtered grid homology
}
\author{Matthew Kendall}
\date{}

\begin{document}
\maketitle

\begin{abstract}
    We extend the crossing-change maps between grid complexes, defined by Ozsv\'ath--Szab\'o--Stipsicz, to filtered grid complexes and give a combinatorial formulation of the Alishahi--Eftekhary $\mathfrak{l}(K)$ invariant.
\end{abstract}

\section{Introduction}
Grid homology is a combinatorial theory for knots and links first developed by Manolescu, Ozsv\'ath, and Sarkar \cite{manolescu-ozsvath-sarkar}, inspired by the ideas of Sarkar and Wang \cite{sarkar-wang}.
It is isomorphic to knot Floer homology, an invariant for knots in three-manifolds discovered in 2003 by Ozsv\'ath and Szab\'o \cite{OS04} and independently by Rasmussen \cite{rasmussen-phd-thesis}.
Grid homology can be used to prove the Milnor conjecture \cite{milnor-cplx-hypersurfaces}, which states that the unknotting number of the $(p,q)$ torus knot is $\frac{(p-1)(q-1)}{2}$.
This conjecture was first verified by Kronheimer and Mrowka \cite{kron-mrow} in 1993 using smooth four-manifold topology.
Inspired by the proof of Sarkar \cite{sarkar}, Ozsv\'ath, Szab\'o, and Stipsicz in their book \cite{grid-homology} use grid homology to define the $\tau$ invariant and compute $\tau$ for torus knots.
They use a definition of $\tau$ relies on pairs of maps of grid complexes whose underlying knots differ by a crossing change.
The goal of this note is to show that these crossing change maps by Ozsv\'ath--Stipsicz--Szab\'o \cite{grid-homology} can be generalized to filtered grid complexes and used to give a combinatorial formulation of the Alishahi--Eftekhary $\mathfrak{l}(K)$ invariant \cite{AE20}, compare also \cite{rasmussen-phd-thesis}.

To a grid diagram $\GG$ of a knot $K \subset S^3$, the {\it grid complex} $GC^-(\GG)$ is a chain complex over the polynomial ring $\FF[V_1,\ldots,V_n]$, where the field $\FF = \ZZ/2\ZZ$.
It is equipped with a Maslov and Alexander bigrading.
Similarly, the {\it filtered grid complex} $\mathcal{GC}^-(\GG)$ is a chain complex over $\FF[V_1,\ldots,V_n]$ equipped with the same Maslov grading, but whose Alexander function on grid states induces a filtration rather than a grading.
The filtered grid complex is a refinement of the standard grid complex in the sense that $GC^-(\GG)$ is the associated graded object of $\mathcal{GC}^-(\GG)$.
Its key property is that the filtered chain homotopy type of $\mathcal{GC}^-(\GG)$ depends only on the underlying knot $K$.

Let $K_+$ be a knot with a distinguished positive crossing and $K_-$ be the knot with the crossing changed.
Represent $K_+$ and $K_-$ by the grid diagrams $\mathbb{G}_+$ and $\mathbb{G}_-$ that differ by a cross-commutation of columns, see Definition~\ref{def:cross-comm}.
We first state the proposition that we wish to refine, compare Section 6.2 of \cite{grid-homology}.

\begin{proposition}\label{prop:6.1.1}
There exist chain maps
\[
    \text{$C_- : GC^-(\GG_+) \to GC^-(\GG_-)$ and $C_+ : GC^-(\GG_-) \to GC^-(\GG_+)$}
\]
where $C_-$ and $C_+$ is homogeneous of degree $(0,0)$ and $(-2,-1)$, such that $C_- \circ C_+$ and $C_+ \circ C_-$ are filtered chain homotopy to multiplication by $V_1$.
\end{proposition}

The filtered analog to Proposition~\ref{prop:6.1.1} that we prove is the following:

\begin{theorem}\label{thm:filt_6.1.1}
There exist filtered chain maps
\[
    \mathcal{C}_- : \mathcal{GC}^-(\mathbb{G}_+) \to \mathcal{GC}^-(\mathbb{G}_-) \;\; \text{and} \;\; \mathcal{C}_+ : \mathcal{GC}^-(\mathbb{G}_-) \to \mathcal{GC}^-(\mathbb{G}_+)
\]
where $\mathcal{C}_{-}$ and $\mathcal{C}_+$ are homogeneous of degree $(0,0)$ and $(-2,-1)$ respectively, such that $\mathcal{C}_- \circ \mathcal{C}_+$ and $\mathcal{C}_+ \circ \mathcal{C}_-$ are filtered chain homotopic to multiplication by $V_1$.
\end{theorem}

Using the fact that $GC^-(\GG)$ is the associated graded object of $\mathcal{GC}^-(\GG)$, it is not difficult to show that Theorem~\ref{thm:filt_6.1.1} implies Proposition~\ref{prop:6.1.1}, see Section~\ref{sec:proof} for details.

Using the crossing change maps in Theorem~\ref{thm:filt_6.1.1}, we obtain a knot invariant $\mathfrak{l}(K)$ which is a lower bound on the unknotting number, see Theorem~\ref{thm:unknotting}.
Moreover, we show that $\mathfrak{l}(K)$ is equal to the Alishahi--Eftekhary knot invariant see Theorem~\ref{thm:l=lAE}.
The Alishahi--Eftekhary invariant, given in Definition 3.1 of \cite{AE20}, is defined using a generalization of sutured Floer homology \cite{AE-suture-floer}, first developed by Juh\'asz \cite{juhasz}. 
In particular, this paper provides a combinatorial interpretation of their invariant.
In the cases when the filtered chain homotopy type $\mathcal{GC}^-(K)$ of the filtered grid complex is known, computations of $\mathfrak{l}(K)$ are possible, see Section~\ref{sec:invariant} for more details.

The paper is organized as follows.
In Section~\ref{sec:def-constructions}, we review the necessary constructions from grid homology.
In Section~\ref{sec:proof}, we prove Theorem~\ref{thm:filt_6.1.1}.
In Section~\ref{sec:invariant}, we define a knot invariant based on the crossing change maps in Theorem~\ref{thm:filt_6.1.1} and prove that the invariant is a lower bound on the unknotting number.
In Section~\ref{sec:AE}, we review the construction of the Alishahi--Eftekhary knot invariant, show the two invariants are equal.

\subsection*{Acknowledgements}
The author would like to thank Peter Ozsv\'ath for his constant guidance and support throughout the project.
The author would also like to thank Ollie Thakar and Isabella Khan for helpful discussions.

\section{Definitions and constructions}\label{sec:def-constructions}
In this section, we review the necessary definitions and constructions from grid homology.
We follow Ozsv\'ath--Szab\'o--Stipsicz's book \cite{grid-homology}.

\subsection{Grid diagrams}
In this section, we review grid diagrams and an operation relating two grids whose underlying knots differ by a crossing change, called cross-commutation. 

\begin{definition}
    A {\it planar grid diagram} $\GG$ is an $n \times n$ grid on the plane with $n$ of the squares marked with an $X$ and $n$ of the squares marked with an $O$.
    The markings are subject to the following rules:
    \begin{enumerate}
        \item Each row has exactly one square marked with an $X$ and a single square marked with an $O$.
        Each column also has exactly one square marked with an $X$ and exactly one square marked with an $O$.

        \item No square is marked with an $X$ and an $O$.
    \end{enumerate}
    The number $n$ is called the {\it grid number} of $\GG$.
\end{definition}

A grid diagram $\GG$ specifies an oriented link $L$ by the following steps:
\begin{enumerate}
    \item Draw oriented segments from the $X$-marked squares to the $O$-marked squares in each column.

    \item Draw oriented segments connecting the $O$-marked squares to the $X$-marked squares in each row such that the vertical segments always cross above the horizontal segments.
\end{enumerate}

Every oriented link in $S^3$ can be represented by a grid diagram, see Lemma 3.1.3 in \cite{grid-homology} for a proof.
For example, Figure~\ref{fig:right-hand-trefoil-grid-diag} is a grid diagram for the right-handed trefoil.

\begin{figure}
    \centering
    \includegraphics[scale=0.5]{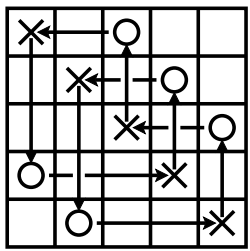}
    \caption{A grid diagram for the right-handed trefoil.}
    \label{fig:right-hand-trefoil-grid-diag}
\end{figure}

The only operation on grid diagrams we will be using is cross-commutation.
\begin{definition}\label{def:cross-comm}
    Fix two consecutive columns (resp. rows) in a grid diagram $\GG$ and let $\GG'$ be obtained by interchanging those two columns (resp. rows).
    Suppose that the interiors of their corresponding intervals intersect nontrivially, but neither is contained in the other.
    Then $\GG,\GG'$ are said to be related by a {\it cross-commutation}.
\end{definition}

See Figure~\ref{fig:cross-comm} for an example of a cross-commutation move.

\begin{figure}
    \centering
    \includegraphics[scale=0.75]{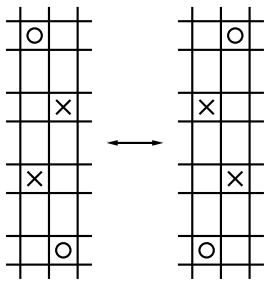}
    \caption{A cross-commutation move.}
    \label{fig:cross-comm}
\end{figure}

Let $K_+$ be a knot with a distinguished positive crossing, and let $K_-$ be the same knot with the distinguished crossing changed.
Then we can represent $K_+$ and $K_-$ with grid diagrams $\GG_+$ and $\GG_-$ differing by a cross-commutation of columns.

Before defining the grid complex, we need to discuss the generators of our complex, called grid states, and their properties.

\subsection{Grid states and connecting grid states}

Consider a toroidal grid diagram $\GG$ for a knot $K$ with grid number $n$.
We can label the horizontal circles $\boldsymbol{\alpha} = \{\alpha_1,\ldots,\alpha_n\}$ and the vertical circles $\boldsymbol{\beta} = \{\beta_1,\ldots,\beta_n\}$.
Define a {\it grid state} to be an $n$-tuple of points $\vec{x} = \{x_1,\ldots,x_n\}$ such that each horizontal circle contains exactly one of the elements in $\vec{x}$ and each vertical circle containe exactly one of the elements in $\vec{x}$.
Let $\mathbf{S}(\GG)$ be the set of grid states for $\GG$.

\begin{definition}\label{def:rectangle}
Fix two grid states $\vec{x},\vec{y} \in \mathbf{S}(\GG)$.
An embedded disk $r$ in the torus whose boundary is the union of four arcs, each of which lies in some $\alpha_j$ or $\beta_j$, is called a {\it rectangle $\vec{x}$ to $\vec{y}$} if it satisfies conditions
\begin{itemize}
    \item At any of the corner points $x$ of $r$, the rectangle contains exactly one of the four quadrants determined by the two intersecting curves at $x$.
    
    \item All of the corner points of $r$ belong to $\vec{x} \cup \vec{y}$.
    
    \item Let $\p_a r$ be the part of the boundary of $r$ belonging to $\alpha_1 \cup \cdots \cup \alpha_n$.
    Then $\p(\p_a r) = \vec{y} - \vec{x}$, where $\vec{y}-\vec{x}$ is thought of as a formal sum of points.
\end{itemize}
Denote the set of rectangles from $\vec{x}$ to $\vec{y}$ by $\Rect(\vec{x},\vec{y})$.
A rectangle is called an {\it empty rectangle} if
\[
    \vec{x} \cap \inter(r) = \vec{y} \cap \inter(r) = \emptyset.
\]
The set of empty rectangles from $\vec{x}$ to $\vec{y}$ is denoted $\Rect^{\circ}(\vec{x},\vec{y})$.
\end{definition}

\subsection{Unblocked grid complex}
For the rest of the paper, we will be working over the field of two elements $\FF = \ZZ/2\ZZ$.

\begin{definition}
Let $\mathbb{G}$ be a toroidal grid diagram with grid number $n$ representing a knot $K$.
The {\it unblocked grid complex} $GC^{-}(\mathbb{G})$ is a free module over $\FF[V_1,\ldots,V_n]$ generated by the grid states $\mathbf{S}(\mathbb{G})$ equipped with a $\FF[V_1,\ldots,V_n]$-module endomorphism $\p_{\XX}^-$ sending $\vec{x} \in \mathbf{S}(\mathbb{G})$ to
\[
    \p^{-}_{\mathbb{X}}(\vec{x}) = \sum_{\vec{y} \in \mathbf{S}(\mathbb{G})} \sum_{\{ r \in \Rect^{\circ}(\vec{x},\vec{y}) \mid r \cap \mathbb{X} = \emptyset \}} V_1^{O_1(r)}\cdots V_n^{O_n(r)} \cdot \vec{y}.
\]
\end{definition}

The endomorphism $\p^{-}_{\mathbb{X}}$ satisfies $(\p^{-}_{\mathbb{X}})^2 = 0$, making $GC^-(\GG)$ into a complex, see Chapter 4 of \cite{grid-homology}.
The unblocked grid complex is bigraded with the Maslov grading and the Alexander grading.
We define these gradings now and give some properties.

\begin{proposition}
The {\it Maslov grading} $M = M_{\mathbb{O}}$ is a unique function on grid states $\mathbf{S}(\mathbb{G})$ uniquely characterized by
\begin{enumerate}
    \item [(M-1)] If $\vec{x}^{NWO}$ is the grade state whose components are the upper left corners of squares marked $O$, then \[M(\vec{x}^{NWO}) = 0.\]
    
    \item [(M-2)] If $\vec{x}$ and $\vec{y}$ can be connected by a rectangle $r \in \Rect(\vec{x},\vec{y})$, then
    \[
        M(\vec{x}) - M(\vec{y}) = 1 - 2\#(r \cap \mathbb{O}) + 2\#(\vec{x} \cap \inter(r)).
    \]
\end{enumerate}
\end{proposition}
The existence and uniqueness of the Maslov grading is shown in Section 4.3 of \cite{grid-homology}.
Define $M_{\mathbb{X}}$ analogously, replacing all instances of $O$ with $X$.

There is another characterization of $M$. 
Consider the partial ordering on $\RR^2$, specified by $(p_1,p_2) < (q_1,q_2)$ if $p_1 < q_1$ and $p_2 < q_2$.
For two finite point-sets $P,Q$, let $I(P,Q)$ be the number of pairs $p \in P$ and $q \in Q$ such that $p < q$.
Let $\mathcal{J}(P,Q) = \frac{1}{2}(I(P,Q) + I(Q,P))$.
Extend $\mathcal{J}$ bilinearly over formal sums and differences of subsets of the plane.
The following is \cite[Lemma 4.3.5]{grid-homology}.
\begin{lemma}\label{lemma:maslov}
\[
    M(\vec{x}) = \mathcal{J}(\vec{x} - \mathbb{O}, \vec{x} - \mathbb{O}) + 1.
\]
\end{lemma}

\begin{definition}\label{def:alexander}
The {\it Alexander function} on grid states is defined by
\[
    A(\vec{x}) = \frac{1}{2}(M_{\mathbb{O}}(\vec{x}) - M_{\mathbb{X}}(\vec{x})) - \left( \frac{n-1}{2}\right).
\]
\end{definition}
The following property characterizes the Alexander grading up to an additive constant, compare Proposition 4.3.3 of \cite{grid-homology}.
\begin{proposition}\label{prop:alexander-grading}
The Alexander function is characterized, up to an additive constant, by the following property.
For any rectangle $r \in \Rect(\vec{x},\vec{y})$,
\[
    A(\vec{x}) - A(\vec{y}) = \#(r \cap \mathbb{X}) - \#(r \cap \mathbb{O}).
\]
\end{proposition}
Extend the Maslov and Alexander functions to the basis $V_1^{k_1}\cdots V_n^{k_n}$ for arbitrary nonnegative integers $k_1,\ldots,k_n$ by defining $M(V_1^{k_1}\cdots V_n^{k_n} \cdot \vec{x}) = M(\vec{x}) - 2\sum_{i=1}^n k_i$ and $A(V_1^{k_1}\cdots V_n^{k_n} \cdot \vec{x}) = A(\vec{x}) - \sum_{i=1}^n k_i$.

In particular, multiplication by $V_i$ has chain map of degree $(-2,-1)$.

\subsection{Filtered grid complexes}
Let $k$ be a field, which in our case will always be $\ZZ/2\ZZ$.

\begin{definition}\label{def:filtered-cx}
A {\it $\ZZ$-filtered, $\ZZ$-graded chain complex} over $k[V_1,\ldots,V_n]$ is $k$-module $\mathcal{C}$ with
\begin{itemize}
    \item A {\it differential} $\p : \mathcal{C} \to \mathcal{C}$, i.e. a $k[V_1,\ldots,V_n]$-module homomorphism with $\p^2 = 0$ compatible with the {\it $\ZZ$-grading} $\mathcal{C} = \bigoplus_{d \in \ZZ} \mathcal{C}_d$, in the sense that $\p \mathcal{C}_d \subseteq \mathcal{C}_{d-1}$.
    
    \item A {\it $\ZZ$-filtration} $\mathcal{F}_s\mathcal{C}$ for $s \in \ZZ$ that exhausts $\mathcal{C}$, in the sense that $\bigcup_{s \in \ZZ} \mathcal{F}_s\mathcal{C} = \mathcal{C}$.
    The filtration is compatible with the $\ZZ$-grading, in the sense that if $\mathcal{F}_s\mathcal{C}_d = \mathcal{F}_s \mathcal{C} \cap \mathcal{C}_d$, then $\mathcal{F}_s\mathcal{C} = \bigcup_{s \in \ZZ} \mathcal{F}_s\mathcal{C}_d$.
    The filtration is compatible with the differential, in the sense that $\p (\mathcal{F}_s \mathcal{C}) \subseteq \mathcal{F}_s \mathcal{C}$, making $\mathcal{F}_s \mathcal{C}$ into a subcomplex.
    The filtration is bounded below, in the sense that for any integer $d$, there exists an integer $n_d$ such that $\mathcal{F}_s \mathcal{C}_d = 0$ for all $s \le n_d$.
    
    \item There are $k$-module endomorphisms $V_i : \mathcal{C} \to \mathcal{C}$, which satisfy: $V_i,V_j$ commute for all $1 \le i,j \le n$, $\p V_i = V_i \p$, $V_i(\mathcal{C}_d) \subseteq \mathcal{C}_{d-2}$, and $V_i(\mathcal{F}_s\mathcal{C}) \subseteq \mathcal{F}_{s-1}\mathcal{C}$.
\end{itemize}
\end{definition}

\begin{definition}
Let $\mathbb{G}$ be a toroidal grid diagram with grid number $n$ representing a knot $K$.
The \emph{filtered grid complex} $\mathcal{GC}^-(\mathbb{G})$ is generated over $\FF[V_1,\ldots,V_n]$ by $\mathbf{S}(\mathbb{G})$ and has differential
\[
    \p^-\vec{x} = \sum_{\mathbf{y} \in \mathbf{S}(\mathbb{G})} \sum_{r \in \Rect^{\circ}(\vec{x},\vec{y})} V_1^{O_1(r)}\cdots V_n^{O_n(r)} \cdot \vec{y}.
\]
The {\it Maslov grading} is given by $M(\vec{x}) = d$ and the {\it Alexander filtration} is given by $A(\vec{x}) = s$, in the sense that $\mathcal{F}_s\mathcal{C}$ is the $\FF$-span of generators $V_1^{k_1}\cdots V_n^{k_n} \cdot \vec{x}$ which evaluate to most $s$ under the Alexander function.
\end{definition}

Note that the differential now counts rectangles containing $X$-markings.
The filtered module $\mathcal{GC}^-(\GG)$ is shown to be a chain complex in Chapter 13 of \cite{grid-homology}.
The most important property is that the filtered quasi-isomorphism type of $\mathcal{GC}^-(\GG)$ depends only on the underlying knot $K$, see Theorem 13.2.9 of \cite{grid-homology}.

Given a $\ZZ$-graded, $\ZZ$-filtered chain complex $(\mathcal{C},\p)$, the {\it associated graded object} is the chain complex $\gr(\mathcal{C}) = \bigoplus_{d,s \in \ZZ} (\mathcal{F}_s\mathcal{C}_d/\mathcal{F}_{s-1}\mathcal{C}_d)$ with bigrading $\gr(\mathcal{C})_{d,s} = \mathcal{F}_s\mathcal{C}_d/\mathcal{F}_{s-1}\mathcal{C}_d$ and differential $\gr(\p)$ induced by $\p : \mathcal{C} \to \mathcal{C}$, i.e. $\gr(\p) = \sum_{d,s} \gr(\p)_{d,s}$ where $\gr(\p)_{d,s} : \mathcal{F}_s\mathcal{C}_d / \mathcal{F}_{s-1}\mathcal{C}_d \to \mathcal{F}_{s}\mathcal{C}_{d-1} / \mathcal{F}_{s-1}\mathcal{C}_{d-1}$.
The following is Proposition 13.2.6 of \cite{grid-homology}.

\begin{proposition}\label{prop:assoc-graded}
    The associated graded object of $(\mathcal{\GC}^{-}(\GG),\p^-)$ is $(GC^-(\GG),\p^-_{\XX})$.    
\end{proposition}

\section{Proof of Theorem~\ref{thm:filt_6.1.1}}\label{sec:proof}
First we show that Theorem~\ref{thm:filt_6.1.1} is a generalization of the Ozsv\'ath--Stipsicz--Szab\'o crossing maps in Proposition~\ref{prop:6.1.1}.
Recall that $(GC^{-}(\mathbb{G}),\p^-_{\mathbb{X}})$ is the associated graded object of $(\mathcal{GC}^-(\mathbb{G}),\p^-)$, see Proposition~\ref{prop:assoc-graded}. 
By functoriality we retrieve two maps
\begin{equation}
    c_- : GC^-(\mathbb{G}_+) \to GC^-(\mathbb{G}_-) \;\; \text{and} \;\; 
    c_+ : GC^-(\mathbb{G}_-) \to GC^-(\mathbb{G}_+)
\end{equation}
such that $c_-$ is bigraded and $c_+$ is homogeneous of degree $(-2,-1)$ such that $c_- \circ c_+$ and $c_+ \circ c_-$ are chain-homotopic to multiplication by $V_1$.
Taking homology, we get the statement of Proposition~\ref{prop:6.1.1}.

Now we proceed to the proof of Theorem~\ref{thm:filt_6.1.1}.
We use a similar set up as in Chapter 5 of \cite{grid-homology}.
Draw both grids $\GG_+$ and $\GG_-$ on the same torus, where the $\OO$ and $\XX$ markings are fixed.
Start with $n$ horizontal circles, $n-1$ vertical circles, and two additional circles $\beta_i$ and $\gamma_i$ corresponding to the $i$-th circles of $\GG_+$ and $\GG_-$, such that $\beta_i$ and $\gamma_i$ intersect in four points.
To define our crossing change maps, we need to mark two of these intersection points, which we explain now.
The complement of $\beta_i \cup \gamma_i$ in the grid torus has five components, four of which are bigons, and each bigon contains exactly one $X$ or exactly one $O$.
Since $\GG_+$ and $\GG_-$ differ by a cross-commutation, the two $X$ marked bigons share a vertex on $\beta_i \cap \gamma_i$; call this vertex $t$.
Label the two $O$-markings such that $O_1$ is above $O_2$ in our grid diagram.
The bigon containing $O_2$ and one of the $X$-labeled bigons share a vertex which we call $s$.
These notational choices are shown in Figure~\ref{fig:crossing-change-grid}.

To define the crossing-change maps, we need to count other domains between grid states: pentagons and hexagons.

\begin{figure}
    \centering
    \includegraphics[scale=0.5]{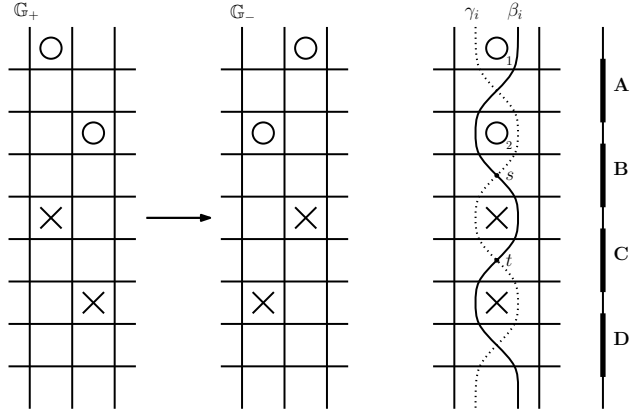}
    \caption{Grid diagram for crossing changes.}
    \label{fig:crossing-change-grid}
\end{figure}

\begin{definition}\label{def:pentagon}
Fix grid states $\vec{x}_- \in \mathbf{S}(\GG_-)$ and $\vec{x}_+ \in \mathbf{S}(\GG_+)$.
An embedded disk $p$ in the torus whose boundary is the union of five arcs, each of which lies in some $\alpha_j,\beta_j$, or $\gamma_i$, is called a {\it pentagon from $\vec{x}_+$ to $\vec{y}_-$} if it satisfies the following conditions
\begin{itemize}
    \item At any of the corner points $x$ of $p$, the pentagon contains exactly one of the four quadrants determined by the two intersecting curves at $x$.
    
    \item Four of the corners of $p$ are in $\vec{x}_+ \cup \vec{y}_-$, the other corner point is chosen from the four intersection points of the curves $\beta_i$ and $\gamma_i$.
    
    \item Let $\p_{\alpha} p$ be the part of the boundary of $p$ belonging to $\alpha_1 \cup \cdots \cup \alpha_n$.
    Then $\p(\p_{\alpha} p) = \vec{y}_- - \vec{x}_+$.
\end{itemize}
Let $\Pent(\vec{x}_+,\vec{y}_-)$ be the set of pentagons from $\vec{x}_+ \in \mathbf{S}(\GG_+)$ to $\vec{y}_- \in \mathbf{S}(\GG_-)$.
A pentagon $p \in \Pent(\vec{x}_+,\vec{y}_-)$ is {\it empty} if $\inter(p) \cap \vec{x}_+ = \inter(p) \cap \vec{y}_- = \emptyset$.
Let $\Pent^{\circ}(\vec{x}_+,\vec{y}_-)$ be the set of empty pentagons from $\vec{x}_+$ to $\vec{y}_-$.
Let $\Pent_s^{\circ}(\vec{x}_+,\vec{y}_-)$ be the set of empty pentagons from $\vec{x}_+$ to $\vec{y}_-$ with one vertex at $s$.
For two grid states $\vec{x}_- \in \mathbf{S}(\GG_-)$ and $\vec{y}_+ \in \mathbf{S}(\GG_+)$, we define $\Pent(\vec{x}_-,\vec{y}_+), \Pent^{\circ}(\vec{x}_-,\vec{y}_+), \Pent^{\circ}_{t}(\vec{x}_-,\vec{y}_+)$ similarly.
\end{definition}

\begin{definition}\label{def:hexagon}
Fix grid states $\vec{x},\vec{y} \in \mathbf{S}(\GG)$.
An embedded disk $h$ in the torus whose boundary is in the union of the $\alpha_j,\beta_j$ (for $j = 1,\ldots,n$) and $\gamma_i$ is called a {\it hexagon from $\vec{x}$ to $\vec{y}$} if it satisfies the following conditions:
\begin{itemize}
    \item At any of the six corner points $x$ of $h$, the hexagon contains exactly one of the four quadrants determined by the two intersecting curves at $x$.
    
    \item Four of the corner points of $h$ are in $\vec{x} \cup \vec{y}$, and the other two corners are $a$ and $b$, where $a,b$ are chosen from the four intersection points of the curves $\beta_i$ and $\gamma_i$.
    
    \item Let $\p_{\alpha} h$ be the part of the boundary of $h$ belonging to $\alpha_1 \cup \cdots \cup \alpha_n$.
    Then $\p(\p_{\alpha} h) = \vec{y} - \vec{x}$.
\end{itemize}
Let $\Hex(\vec{x},\vec{y})$ be the set of hexagons from $\vec{x}$ to $\vec{y}$.
A hexagon $h \in \Hex(\vec{x},\vec{y})$ is {\it empty} if $\inter(h) \cap \vec{x} = \inter(h) \cap \vec{y} = \emptyset$.
Let $\Hex^{\circ}(\vec{x},\vec{y})$ be the set of empty hexagons from $\vec{x}$ to $\vec{y}$.
Let $\Hex_{s,t}^{\circ}(\vec{x},\vec{y})$ denote the set of empty hexagons with two consecutive corners at $s$ and $t$ in the order specified by the orientation of the hexagon.
Let $\Hex_{t,s}^{\circ}(\vec{x},\vec{y})$ be the analogous set with the order of $s$ and $t$ reversed.
\end{definition}

Now we define the crossing change maps.
Fix arbitrary grid states $\vec{x}_+ \in \vec{S}(\GG_+)$ and $\vec{x}_- \in \vec{S}(\GG_-)$.
Recall that $\Pent_s^{\circ}(\vec{x}_+,\vec{x}_-)$ is the set of empty pentagons from $\vec{x}_+$ to $\vec{x}_-$ with one vertex at $s$, and $\Pent_t^{\circ}(\vec{x}_-,\vec{x}_+)$ is the set of empty pentagons from $\vec{x}_-$ to $\vec{x}_+$ with one vertex at $t$.
Define the two $\FF[V_1,\ldots,V_n]$-module maps
\[
    \mathcal{C}_- : \mathcal{GC}^-(\mathbb{G}_+) \to \mathcal{GC}^-(\mathbb{G}_-) \;\; \text{and} \;\; \mathcal{C}_+ : \mathcal{GC}^-(\mathbb{G}_-) \to \mathcal{GC}^-(\mathbb{G}_+)
\]
by counting pentagons either containing the vertex $s$ or the vertex $t$:
\begin{align}
    \mathcal{C}_-(\vec{x}_+) &= \sum_{\vec{y}_- \in \vec{S}(\mathbb{G}_-)} \sum_{p \in \Pent_s^{\circ}(\vec{x}_+,\vec{y}_-)} V_1^{O_1(p)}\cdots V_n^{O_n(p)} \cdot \vec{y}_-, \label{eq:C_-} \\
    \mathcal{C}_+(\vec{x}_-) &= \sum_{\vec{y}_+ \in \vec{S}(\mathbb{G}_+)} \sum_{p \in \Pent_t^{\circ}(\vec{x}_-,\vec{y}_+)} V_1^{O_1(p)}\cdots V_n^{O_n(p)} \cdot \vec{y}_+ \label{eq:C_+}.
\end{align}
Here $O_i(p)$ is $1$ if $p$ contains $O_i$ and $0$ otherwise, where $O_1,\ldots,O_n$ are the $\OO$ markings.

For fixed grid states $\vec{x}_{\pm} \in \GG_+$ and $\vec{y}_{\mp} \in \GG_-$, a {\it domain} $\psi$ from $\vec{x}_{\pm}$ to $\vec{y}_-$ is a formal sum of closures of regions in the complement of the $\alpha_j,\beta_j$, and $\gamma_i$ in the grid torus, taken with integral multiplicities, such that $\p_{\alpha} \psi$, the portion of the boundary in $\alpha_1 \cup \cdots \cup \alpha_n$, satisfies $\p(\p_{\alpha} \psi) = \vec{y}_{\mp} - \vec{x}_{\pm}$.
Let $\pi(\vec{x}_{\pm},\vec{y}_{\mp})$ be the set of domains from $\vec{x}_{\pm}$ to $\vec{y}_{\mp}$.

% \begin{theorem}
% The $\FF[V_1,\ldots,V_n]$-module maps $\mathcal{C}_-$ and $\mathcal{C}_+$ satisfy
% \begin{itemize}
%     \item $\mathcal{C}_{-}$ is a $\ZZ$-graded, $\ZZ$-filtered chain map and $\mathcal{C}_{+}$ is a homogeneous chain map of degree $(-2,-1)$.

%     \item $\mathcal{C}_- \circ \mathcal{C}_+$ is filtered chain homotopic to $V_1$, as an operator on $\mathcal{GC}^-(\mathbb{G}_-)$, and $\mathcal{C}_+ \circ \mathcal{C}_-$ is filtered chain-homotopic to multiplication by $V_1$, as an operator on $\mathcal{GC}^-(\mathbb{G}_+)$.
% \end{itemize}
% \end{theorem}

\begin{lemma}\label{lemma:chain-maps}
$\mathcal{C}_-$ and $\mathcal{C}_+$ are chain maps.
\end{lemma}
\begin{proof}
This proof follows Lemma 5.1.4 in \cite{grid-homology} with some extra cases.

To show that $\mathcal{C}_-$ is a chain map, since we are working over $\FF = \ZZ/2\ZZ$, it is enough to show that $\p^- \circ \mathcal{C}_- + \mathcal{C}_- \circ \p^-$ is identically zero.
This expression can be written as
\begin{equation}\label{eq:chain-maps}
\p^- \circ \mathcal{C}_- + \mathcal{C}_- \circ \p^-
= \sum_{\vec{z}_- \in \mathbf{S}(\GG)} \sum_{\psi \in \pi(\vec{x}_+,\vec{z}_-)} N(\psi) \cdot V_1^{O_1(\psi)} \cdots V_n^{O_n(\psi)} \cdot \vec{z}_-,
\end{equation}
where $N(\psi)$ is the number of ways to decompose $\psi$ as either $r * p$ or $p' * r'$, where $r,r'$ are empty rectangles and $p,p'$ are empty pentagons.
There are three cases of $\psi \in \pi(\vec{x}_+,\vec{z}_-)$:
\begin{itemize}
    \item [(P-1)] $\vec{x}_+ \setminus (\vec{x}_+ \cap \vec{z}_-)$ consists of $4$ points. 
    In this case, there are two decompositions of $\psi$ with the same underlying rectangle and pentagon, only differing in the grid states they connect. 
    See case (P-1) in Figure~\ref{fig:domain-decomp}.
    Thus,
    $N(\psi) = 2$, so there are no contribution of terms in this case. 

    \begin{figure}
    \centering
    \includegraphics[scale=0.5]{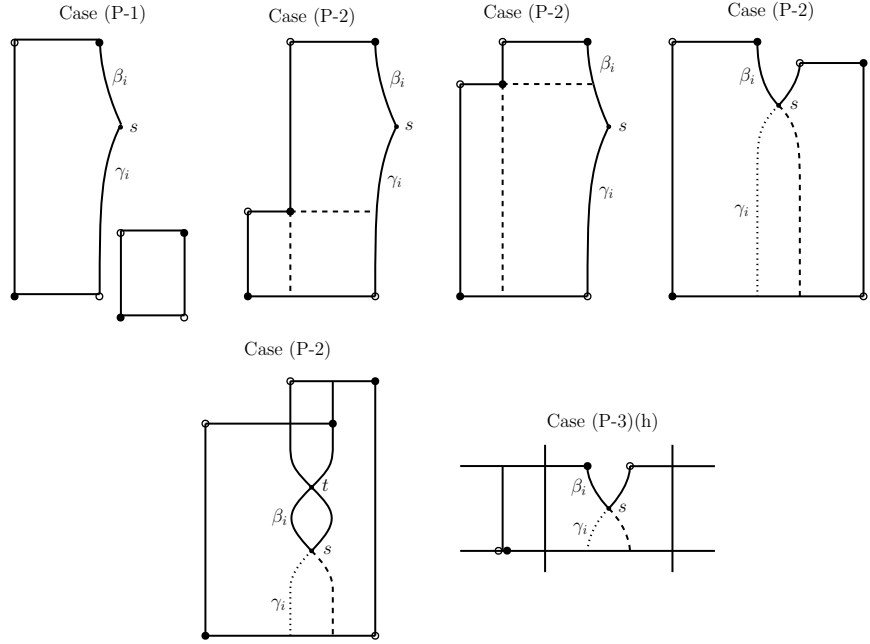}
    \caption{Domain decompositions. The black dots are contained in $\vec{x}_+$ and the white dots are contained in $\vec{z}_-$.}
    \label{fig:domain-decomp}
    \end{figure}
    
    \item [(P-2)] $\vec{x}_+ \setminus (\vec{x}_+ \cap \vec{z}_-)$ consists of $3$ points.
    There are two cases to consider here: either all of the local multiplicites of $\psi$ are $0$ and $1$, or some local multiplicity is $2$.
    In the first case, $\psi$ has seven corners, one of them being a $270^{\circ}$ corner.
    Cutting this corner in two directions gives two different decompositions of $\psi$ as a rectangle and a pentagon.
    In the second case, $\psi$ has a $270^{\circ}$ corner at $a$, and cutting it in two ways gives two decompositions of $\psi$ into a rectangle and a pentagon.
    See case (P-2) in Figure~\ref{fig:domain-decomp}.
    In all cases, $N(\psi)=2$, so there are no contribution of terms in this case.
    
    \item [(P-3)] $\vec{x}_+ \setminus (\vec{x}_+ \cap \vec{z}_-)$ consists of $1$ point.
    In this case, $\vec{z}_-$ is the unique grid state which agrees with $\vec{x}_+$ in all but the component $\beta_i$, and the domain $\psi$ goes around the torus, either horizontally or vertically.
    In the horizontal case, the domain $\psi$ is a horizontal thin annulus minus a small triangle, which can be decomposed in two ways.
    See case (P-3)(h) Figure~\ref{fig:domain-decomp}.
    
    In the vertical case, the decomposition is unique.
    Luckily, we can pair the off domains $\psi$ depending on whether their support lies between $\beta_{i-1}$ and $\beta_i$, or between $\beta_{i}$ and $\beta_{i+1}$.
    
    There are two thin annular regions $A_1$ and $A_2$ which have three corners: one corner at $s$, another corner is at $\vec{x}_+ \cap \beta_i$ and another corner is at $\vec{x}_-$.
    There are four combinatorial types of $A_1$ depending on which bigon $\vec{x}_+ \cap \beta_i$ lies on.
    We show the four decompositions in Figure~\ref{fig:vert-annulus}.

    \begin{figure}
        \centering
        \includegraphics[scale=0.5]{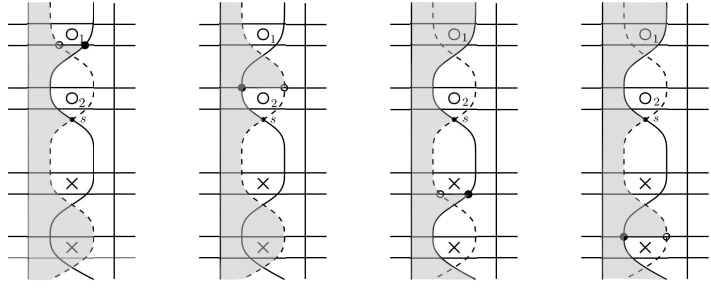}
        \caption{Vertical annulus depending on location of $\vec{x}_+ \cap \beta_i$.}
        \label{fig:vert-annulus}
    \end{figure}
    
    The annulus $A_1$ can be decomposed uniquely as a rectangle and a pentagon in an order determined by the position of $\vec{x}_+ \cap \beta_{i-1}$ relative to $\vec{x}_+ \cap \beta_i$.
    See Figure~\ref{fig:vert-annulus-decomp} for the decomposition of $A_1$ when $\vec{x} \cap \beta_i$ lies on the $O_2$ bigon. 
    Similarly, the $A_2$ annulus can be uniquely decomposed as the sum of a rectangle and pentagon.
    Since the $A_1$ and $A_2$ annuli cross the same $X$ and $O$ markings, their contributions to equation~\eqref{eq:chain-maps} cancel.

    \begin{figure}
        \centering
        \includegraphics[scale=0.5]{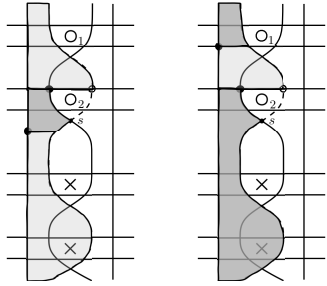}
        \caption{The decomposition of annulus $A_1$ depending on location of $\vec{x}_+ \cap \beta_i$. The darker gray region is the first in the decomposition.}
        \label{fig:vert-annulus-decomp}
    \end{figure}
\end{itemize}

This shows that $\mathcal{C}_-$ is a chain map.
A similar argument can be used to show that $\mathcal{C}_+$ is a chain map.
\end{proof}

Fix two $\ZZ$-filtered, $\ZZ$-graded chain complexes $\mathcal{C}$ and $\mathcal{C}'$ over $\FF[V_1,\ldots,V_n]$.
Call a chain map $f : \mathcal{C} \to \mathcal{C}'$ {\it homogeneous of degree $(m,t)$} if $f(\mathcal{C}_d) \subseteq \mathcal{C}_{d+m}$ and $f(\mathcal{F}_s\mathcal{C}) \subseteq \mathcal{F}_{s+t}\mathcal{C}'$.

\begin{lemma}\label{lemma:bigrading}
$\mathcal{C}_-$ is homogeneous of degree $(0,0)$ and $\mathcal{C}_+$ is homogeneous of degree $(-2,-1)$.
\end{lemma}
\begin{proof}
There is a one-to-one correspondence
\[
    I : \mathbf{S}(\mathbb{G}_-) \to \mathbf{S}(\mathbb{G}_+)
\]
called the {\it nearest-point map} that sends a grid state $\vec{z}_- \in \mathbf{S}(\mathbb{G})$ to the unique grid state $\vec{z}_+ = I(\vec{z}_-)$ that agrees with $\vec{z}_-$ in all but one component (cf. Lemma 5.1.3 of \cite{grid-homology}).

To calculate the grading changes, we associate rectangles to each pentagon via the nearest point map.
Fix a grid state $\vec{x}_+ \in \mathbf{S}(\mathbb{G}_+)$. 
For each grid state $\vec{y}_- \in \mathbf{S}(\mathbb{G}_-)$ such that there exists a pentagon $p \in \Pent_s^{\circ}(\vec{x}_+,\vec{y}_-)$, we consider the associated rectangle $r = r(p) \in \Rect^{\circ}(\vec{x}_+,\vec{y}_+)$.
Define the following constants depending on $\vec{x}_+, \vec{y}_- \in \mathbf{S}(\mathbb{G}_-)$: 
\begin{align*}
    \Delta_r(\vec{x}_+,\vec{y}_-) &= \#(p \cap \mathbb{O}) - \#(r \cap \mathbb{O}), \\
    \Delta_M(\vec{y}_-) &= M(\vec{y}_-) - M(\vec{y}_+), \\
    \Delta_A(\vec{y}_-) &= A(\vec{y}_-) - A(\vec{y}_+).
\end{align*}
First, note that $\Delta_r(\vec{x}_+,\vec{y}_-)$ only depends on the location of $\vec{y}_- \cap \beta_i$: the location of $\vec{x}_+ \cap \beta_i$ does not change the difference $\#(p \cap \OO) - \#(r \cap \OO)$.
So we can drop the dependence on $\vec{x}_+$ and write $\Delta_r(\vec{y}_-) = \Delta_r(\vec{x}_+,\vec{y}_-)$.

For simplicity, let $\Delta_M(\vec{x}_+, \vec{y}_-)$ be the difference $M(V_1^{O_1(p)}\cdots V_n^{O_n(p)} \cdot \vec{y}_-) - M(\vec{x}_+)$, which is the change in Maslov grading.
Similarly define $\Delta_A(\vec{x}_+,\vec{y}_-)$ to be the difference in Alexander grading.
Then we can rewrite the change in Maslov grading in terms of these constants:
\begin{align*}
    \Delta_M(\vec{x}_+, \vec{y}_-) &= 
    M(V_1^{O_1(p)}\cdots V_n^{O_n(p)} \cdot \vec{y}_-) - M(\vec{x}_+) \\
    &= M(\vec{y}_-) - 2\#(p \cap \mathbb{O}) - M(\vec{x}_+) \\
    &= (M(\vec{y}_+) + \Delta_M(\vec{y}_-)) - 2\#(p \cap \mathbb{O}) - M(\vec{x}_+) \\
    &= M(\vec{y}_+) + \Delta_M(\vec{y}_-) - 2(\Delta_r(\vec{y}_-) + \#(r \cap \mathbb{O})) - M(\vec{x}_+) \\
    &= -(1 - 2\#(r \cap \mathbb{O})) + \Delta_M(\vec{y}_-) - 2\Delta_r(\vec{y}_-) - 2\#(r \cap \mathbb{O}) \tag{Property (M-2)} \\
    &= \Delta_M(\vec{y}_-) - 2\Delta_r(\vec{y}_-) - 1. 
\end{align*}
Similarly, we write the change in Alexander grading in terms of these constants:
\begin{align*}
    \Delta_A(\vec{x}_+,\vec{y}_-) &= 
    A(V_1^{O_1(p)}\cdots V_n^{O_n(p)} \cdot \vec{y}_-) - A(\vec{x}_+) \\
    &= A(\vec{y}_-) - \#(p \cap \mathbb{O}) - A(\vec{x}_+) \\
    &= (A(\vec{y}_+) + \Delta_A(\vec{y}_-)) - (\#(r \cap \mathbb{O}) + \Delta_r(\vec{y}_-)) - A(\vec{x}_+) \\
    &= (\#(r \cap \mathbb{O}) - \#(r \cap \mathbb{X})) + \Delta_A(\vec{y}_-) - \#(r \cap \mathbb{O}) - \Delta_r(\vec{y}_-) \\
    &= \Delta_A(\vec{y}_-) - \Delta_r(\vec{y}_-) - \#(r \cap \mathbb{X}). \tag{Proposition~\ref{prop:alexander-grading}}
\end{align*}
Thus, it remains to evaluate $\Delta_r(\vec{y}_-), \Delta_M(\vec{y}_-)$, and $\Delta_A(\vec{y}_-)$ based on the location of $\vec{y}_-$.
To do so, we perform casework based on the location of component $i$ of $\vec{y}_-$. 
We split the $n$ points on the $i$th circle into four special markings $\mathbf{A}, \mathbf{B}, \mathbf{C}, \mathbf{D}$, as shown on the right of Figure~\ref{fig:crossing-change-grid}.
The second, third, and fourth columns of Table~\ref{tab:changes} show the change in $\Delta_r(\vec{y}_-)$.
Note that the change $\Delta_r(\vec{y}_-)$ does not depend on whether we have a right or left rectangle.

To compute the change in the Maslov grading, which is the fourth column of Table~\ref{tab:changes}, we use the equation $\Delta_M(\vec{x}_+, \vec{y}_-) = \Delta_M(\vec{y}_-) - 2\Delta_r(\vec{y}_-) - 1$.
The change in the Alexander grading is the fifth column, which for rows labeled $\mathbf{A},\mathbf{B},\mathbf{D}$ is computed using the upper bound $\Delta_A(\vec{x}_+,\vec{y}_-) \le \Delta_A(\vec{y}_-) - \Delta_r(\vec{y}_-)$ and for row $\mathbf{C}$, we note that $\#(r \cap \mathbb{X}) \ge 1$, so we get a stronger upper bound of $\Delta_A(\vec{x}_+,\vec{y}_-) \le \Delta_A(\vec{y}_-) - \Delta_r(\vec{y}_-) - 1$.
\begin{table}[ht]
    \begin{tabular}{|c|c|c|c|c|c|}
        \hline
         & $\Delta_r(\vec{y}_-)$ & $\Delta_M(\vec{y}_-)$ & $\Delta_A(\vec{y}_-)$ & $\Delta_M(\vec{x}_+,\vec{y}_-)$ & $\Delta_A(\vec{x}_+,\vec{y}_-)$ \\ 
         & $= \#(p \cap \OO) - \#(r \cap \OO)$ & $= M(\vec{y}_-) - M(\vec{y}_+)$ & $= A(\vec{y}_-) - A(\vec{y}_+)$ & & \\
         \hline
         \textbf{A} & $-1$ & $-1$ & $-1$ & $0$ & $\le 0$ \\ \hline
         \textbf{B} & $0$ & $1$ & $0$ & $0$ & $\le 0$ \\ \hline
         \textbf{C} & $0$ & $1$ & $1$ & $0$ & $\le 0$ \\ \hline
         \textbf{D} & $0$ & $1$ & $0$ & $0$ & $\le 0$ \\ \hline
    \end{tabular}
    \caption{Local changes depending on whether $\vec{y}_- \cap \gamma_i$ lies on $\mathbf{A}, \mathbf{B}, \mathbf{C}, \mathbf{D}$.}
    \label{tab:changes}
\end{table}

Now we show the computation for the change $\Delta_M(\vec{y}_-) = M(\vec{y}_-) - M(\vec{y}_+)$.
Let $\mathbb{O}_-,\mathbb{X}_-$ denote the set of $O$'s and $X$'s in $\mathbb{G}_-$, and let $\mathbb{O}_+,\mathbb{X}_+$ be the set of $O$'s and $X$'s in $\mathbb{G}_+$.
\begin{align*}
    \mathcal{J}(\vec{y}_-,\vec{y}_-) - \mathcal{J}(\vec{y}_+,\vec{y}_+) &= 0, \\
    \mathcal{J}(\vec{y}_-, \mathbb{O}_-) - \mathcal{J}(\vec{y}_+, \mathbb{O}_+) &= \begin{cases} 1 & \text{if $\vec{y}_- \cap \gamma_i \in \mathbf{A}$,} \\
    0 & \text{if $\vec{y} \cap \gamma_i \in \mathbf{B} \cup \mathbf{C} \cup \mathbf{D}$,} \end{cases} \\
    \mathcal{J}(\mathbb{O}_-, \mathbb{O}_-) - \mathcal{J}(\mathbb{O}_+, \mathbb{O}_+) &= 1.
\end{align*}
Now using Lemma~\ref{lemma:maslov}, we compute
\begin{align*}
    \Delta_M(\vec{y}_-) = M(\vec{y}_-) - M(\vec{y}_+) 
    &= \mathcal{J}(\vec{y}_- - \mathbb{O}_-, \vec{y}_- - \mathbb{O}_-) - \mathcal{J}(\vec{y}_+ - \mathbb{O}_+, \vec{y}_+ - \mathbb{O}_+) \\
    &= \begin{cases}
        -1 & \text{if $\vec{y}_- \cap \gamma_i \in \mathbf{A}$,} \\
         1 & \text{if $\vec{y} \cap \gamma_i \in \mathbf{B} \cup \mathbf{C} \cup \mathbf{D}$}.
    \end{cases}
\end{align*}
To compute $\Delta_A(\vec{y}_-) = A(\vec{y}_-) - A(\vec{y}_+)$, we compute $M_{\mathbb{X}}(\vec{y}_-) - M_{\mathbb{X}}(\vec{y}_+)$ and use Definition~\ref{def:alexander} of the Alexander function.
Following the same computation as above, we find that
\[
    M_{\XX}(\vec{y}_-) - M_{\XX}(\vec{y}_+)
    = \begin{cases}
    -1 & \text{if $\vec{y}_- \cap \gamma_i \in \mathbf{C}$,} \\
     1 & \text{if $\vec{y} \cap \gamma_i \in \mathbf{A} \cup \mathbf{B} \cup \mathbf{D}$}.
    \end{cases}
\]
Then
\begin{align*}
    \Delta_A(\vec{y}_-) = A(\vec{y}_-) - A(\vec{y}_+)
    &= \frac{1}{2}((M(\vec{y}_-) - M(\vec{y}_+)) - (M_{\XX}(\vec{y}_-) - M_{\XX}(\vec{y}_+)) ) \\
    &= \begin{cases}
    -1 & \text{if $\vec{y}_- \cap \gamma_i \in \mathbf{A}$}, \\
    0 & \text{if $\vec{y}_- \cap \gamma_i \in \mathbf{B} \cup \mathbf{D}$}, \\
    1 & \text{if $\vec{y}_- \cap \gamma_i \in \mathbf{C}$}.
    \end{cases}
\end{align*}

This proves that $\mathcal{C}_-$ is a $\ZZ$-graded, $\ZZ$-filtered chain map.

The argument to show that $\mathcal{C}_+$ is homogeneous of degree $(-2,-1)$ is similar.
In this case, we define the constants
\begin{align*}
    \Delta_r(\vec{y}_+) &= \#(p \cap \mathbb{O}) - \#(r \cap \mathbb{O}), \\
    \Delta_M(\vec{y}_+) &= M(\vec{y}_+) - M(\vec{y}_-), \\
    \Delta_A(\vec{y}_+) &= A(\vec{y}_+) - A(\vec{y}_-),
\end{align*}
and similar computations give that
\begin{align*}
    \Delta_M(\vec{x}_+,\vec{y}_-) &= 
    M(V_1^{O_1(p)} \cdots V_n^{O_n(p)} \cdot \vec{y}_+) - M(\vec{x}_-) = \Delta_M(\vec{y}_+) - 2\Delta_r(\vec{y}_+) - 1, \\
    \Delta_A(\vec{x}_+,\vec{y}_-) &=
    A(V_1^{O_1(p)} \cdots V_n^{O_n(p)} \cdot \vec{y}_+) - A(\vec{x}_-) = \Delta_A(\vec{y}_+) - \Delta_r(\vec{y}_-) - \#(r \cap \XX).
\end{align*}
The table recording each of the above constants depending on the location of $\vec{y}_+ \cap \beta_i$ is shown in Table~\ref{tab:changes-C+}.
For rows {\bf A}, {\bf B}, {\bf D}, we note that $\#(r \cap \mathbb{X}) \ge 1$, so we get a stronger upper bound of $\Delta_A(\vec{x}_-,\vec{y}_+) \le \Delta_A(\vec{y}_+) - \Delta_r(\vec{y}_+) - 1$.
\begin{table}[ht]
    \begin{tabular}{|c|c|c|c|c|c|}
        \hline
         & $\Delta_r(\vec{y}_+)$ & $\Delta_M(\vec{y}_+)$ & $\Delta_A(\vec{y}_+)$ & $\Delta_M(\vec{x}_-,\vec{y}_+)$ & $\Delta_A(\vec{x}_-,\vec{y}_+)$ \\ 
         & $= \#(p \cap \OO) - \#(r \cap \OO)$ & $= M(\vec{y}_+) - M(\vec{y}_-)$ & $= A(\vec{y}_+) - A(\vec{y}_-)$ & & \\
         \hline
         \textbf{A} & $1$ & $1$ & $1$ & $-2$ & $\le -1$ \\ \hline
         \textbf{B} & $0$ & $-1$ & $0$ & $-2$ & $\le -1$ \\ \hline
         \textbf{C} & $0$ & $-1$ & $-1$ & $-2$ & $\le -1$ \\ \hline
         \textbf{D} & $0$ & $-1$ & $0$ & $-2$ & $\le -1$ \\ \hline
    \end{tabular}
    \caption{Local changes depending on whether $\vec{y}_+ \cap \beta_i$ lies on $\mathbf{A}, \mathbf{B}, \mathbf{C}, \mathbf{D}$.}
    \label{tab:changes-C+}
\end{table}
This concludes the proof.
\end{proof}

The remaining part of the proof of Theorem~\ref{thm:filt_6.1.1} closely follows the proof of the Proof of Proposition 6.1.1 and Lemma 5.1.6 in \cite{grid-homology}.

\begin{proof}[Proof of Theorem~\ref{thm:filt_6.1.1}]

By Lemma~\ref{lemma:chain-maps} and Lemma~\ref{lemma:bigrading}, $\mathcal{C}_{-}$ is a $\ZZ$-graded, $\ZZ$-filtered chain map and $\mathcal{C}_{+}$ is a homogeneous chain map of degree $(-2,-1)$.
Now it remains the verify that $\mathcal{C}_- \circ \mathcal{C}_+$ and $\mathcal{C}_+ \circ \mathcal{C}_-$ are filtered chain-homotopic to multiplication by $V_1$.

Define the homotopy operators
\[
    \mathcal{H}_- : \mathcal{GC}^-(\mathbb{G}_-) \to \mathcal{GC}^-(\mathbb{G}_-) \;\; \text{and} \;\; \mathcal{H}_+ : \mathcal{GC}^-(\mathbb{G}_+) \to \mathcal{GC}^-(\mathbb{G}_+)
\]
by counting certain hexagons (see Definition~\ref{def:hexagon}) from a given point:
\begin{align*}
    \mathcal{H}_-(\vec{x}_-) &= \sum_{\vec{y}_- \in \vec{S}(\mathbb{G}_-)} \sum_{h \in \Hex_{s,t}^{\circ}(\vec{x}_-,\vec{y}_-)} V_1^{O_1(h)} \cdots V_n^{O_n(h)} \cdot \vec{y}_-, \\
    \mathcal{H}_+(\vec{x}_+) &= \sum_{\vec{y}_+ \in \vec{S}(\mathbb{G}_+)} \sum_{h \in \Hex_{t,s}^{\circ}(\vec{x}_+,\vec{y}_+)} V_1^{O_1(h)} \cdots V_n^{O_n(h)} \cdot \vec{y}_+.
\end{align*}
We claim that $\mathcal{H}_+$ (resp. $\mathcal{H}_-$) is a chain homotopy operator from $\mathcal{C}_+ \circ \mathcal{C}_-$ (resp. $\mathcal{C}_- \circ \mathcal{C}_+$) to $V_1$.
It is clear that $\mathcal{H}_+$ and $\mathcal{H}_-$ are $\FF[V_1,\ldots,V_n]$-module homomorphisms.

Now we show that $\mathcal{H}_+$ is homogeneous of degree $(-1,0)$, i.e. $\mathcal{H}_+(\mathcal{F}_s\mathcal{C}_d) \subseteq \mathcal{F}_s\mathcal{C}_{d+1}$.
For an empty hexagon $h \in \Hex_{s,t}^{\circ}(\vec{x}_+,\vec{y}_+)$,
there is a corresponding empty rectangle $r$ from $\vec{x}_+$ to $\vec{y}_+$ that contains one more $X$-marking that $h$ and has the same number of $O$-markings as $h$.
Then we can compute
\begin{align*}
    M(V_1^{O_1(h)}\cdots V_1^{O_n(h)} \cdot \vec{y}_+) - M(\vec{x}_+) 
    &= M(\vec{y}_+) - M(\vec{x}_+) - 2\#(h \cap \mathbb{O}) \\
    &= M(\vec{y}_+) - M(\vec{x}_+) - 2\#(r \cap \mathbb{O}) = -1. \tag{Property (M-2)}
\end{align*}
Now we compute the difference in Alexander gradings:
\begin{align*}
    A(V_1^{O_1(h)}\cdots V_1^{O_n(h)} \cdot \vec{y}_+) - A(\vec{x}_+)
    = A(\vec{y}_+) - A(\vec{x}_+) - \#(h \cap \mathbb{O}) = -\#(h \cap \mathbb{X}) \le 0.
\end{align*}
We can similarly show that $\mathcal{H}_-$ is homogeneous of degree $(-1,0)$.

It remains to show the $\mathcal{H}_+$ and $\mathcal{H}_-$ satisfy the homotopy formulas
\begin{align}
    \p^- \circ \mathcal{H}_+ + \mathcal{H}_+ \circ \p^- &= \mathcal{C}_+ \circ \mathcal{C}_- - V_1, \label{eq:homotopy_+} \\
    \p^- \circ \mathcal{H}_- + \mathcal{H}_- \circ \p^- &= \mathcal{C}_- \circ \mathcal{C}_+ - V_1. \label{eq:homotopy_-}
\end{align}
Since we are working over $\FF = \ZZ/2\ZZ$, to show equation~\eqref{eq:homotopy_+}, it will be more convenient to show
\begin{equation}\label{eq:homotopy_+_2}
     \p^- \circ \mathcal{H}_+ + \mathcal{H}_+ \circ \p^- + \mathcal{C}_+ \circ \mathcal{C}_- = V_1.
\end{equation}
The left side of~\eqref{eq:homotopy_+_2} can be expanded as
\[
    (\p^- \circ \mathcal{H}_+ + \mathcal{H}_+ \circ \p^- + \mathcal{C}_+ \circ \mathcal{C}_-)(\vec{x}_+)
    = \sum_{\vec{z}_+ \in \mathbf{S}(\mathbb{G})} \sum_{\psi \in \pi(\vec{x_+},\vec{z}_+)} N(\psi) \cdot V_1^{O_1(\psi)} \cdots V_n^{O_n(\psi)} \cdot \vec{z}_+,
\]
where $N(\psi)$ is the number of ways of decomposing $\psi$ as either:
\begin{itemize}
    \item $\psi = r * h$, where $r$ is an empty rectangle and $h$ is an empty hexagon,
    \item $\psi = h * r$, where $h$ is an empty hexagon and $r$ is an empty rectangle,
    \item $\psi = p * p'$, where $p$ is an empty pentagon from $\GG_+$ to $\GG_-$ and $p'$ is a pentagon from $\GG_-$ to $\GG_+$.
\end{itemize}
There are three cases of $\psi \in \pi(\vec{x}_+,\vec{z}_+)$.
\begin{itemize}
    \item [(H-1)] $\vec{x}_+ \setminus (\vec{x}_+ \cap \vec{z}_+)$ consists of $4$ points.
    The two decompositions of $\psi$ are $r_1 * h_1$ and $h_2 * r_2$, where $r_1$ and $r_2$ are rectangles with the same support and $h_1$ and $h_2$ are hexagons with the same support.
    See case (H-1) of Figure~\ref{fig:hex-domain-decomp}.
    Therefore, $N(\psi) = 2$.

    \item [(H-2)] $\vec{x}_+ \setminus (\vec{x}_+ \cap \vec{z}_+)$ consists of $3$ points.
    In this case, $\psi$ has eight corners. Either seven of the corners are $90^{\circ}$ and one is $270^{\circ}$, or five are $90^{\circ}$ and three are $270^{\circ}$.
    In the first case, cutting at the $270^{\circ}$ corner gives two decompositions of $\psi$, and in the second case, the at one of the corners labeled $s$ or $t$ we can cut in two ways.
    See case (H-2) of Figure~\ref{fig:hex-domain-decomp}.
    In all cases, $N(\psi) = 2$.

    \begin{figure}
    \centering
    \includegraphics[scale=0.5]{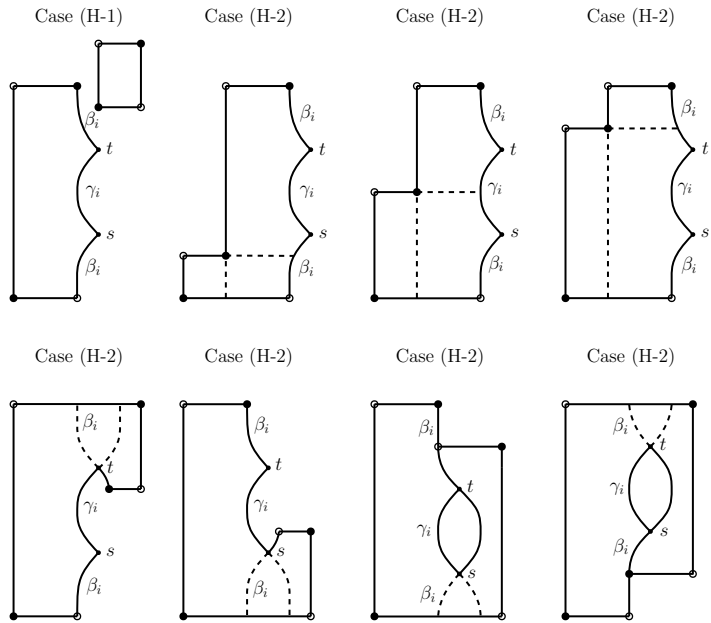}
    \label{fig:hex-domain-decomp}
    \caption{Hexagon domain decompositions. The black dots are contained in $\vec{x}_+$ and the white dots are contained in $\vec{z}_+$.}
    \end{figure}
    
    \item [(H-3)] $\vec{x} = \vec{z}$.
    In this case, $\psi$ is supported inside an annulus between $\beta_i$ and one of the consecutive vertical circles $\beta_{i-1}$ and $\beta_{i+1}$.
    In this case, $N(\psi) = 1$ and decomposes uniquely into one of rectangle-hexagon, hexagon-rectangle, or pentagon-pentagon, depending on the placement of $\vec{x}_+$, see the next paragraph.
    In each case, $\psi$ contains $O_1$, so the left side of~\eqref{eq:homotopy_+_2} contributes $V_1\vec{x}_+$.
    This agrees with the right side of~\eqref{eq:homotopy_+_2}.
    
    To describe the decomposition of $\psi$ more specifically, we perform casework on the placement of $\vec{x}_+ \cap \beta_i$.
    If $\vec{x}_+ \cap \beta_i$ is on the short arc between $s$ and $t$, then the annulus to the east of $\beta_i$ and west of $\beta_{i+1}$ has a unique decomposition.
    If $\vec{x}_+ \cap \beta_i$ is not on the shorter arc connecting $s$ and $t$, then the annulus to the west of $\beta_i$ and to the east of $\beta_{i-1}$ has a unique decomposition (cf. pages 119 and 120 of \cite{grid-homology}).
\end{itemize}

The verification of the homotopy formula for $\mathcal{H}_-$ in~\eqref{eq:homotopy_-} works similarly.
\end{proof}

\section{Crossing-change invariant}\label{sec:invariant}
In this section, we obtain a knot invariant from the crossing change maps in Theorem~\ref{thm:filt_6.1.1} and show that it is a lower bound on the unknotting number.

Let $K$ and $K'$ be two knots, which for our purposes are always in $S^3$.
The {\it Gordian distance} $u(K,K')$ between $K$ and $K'$ is the minimum number of crossing changes required to change $K$ into $K'$.
We can assume there is a planar diagram $K$ such that after switching $m$ crossings, it is isotopic to $K'$.
Label the unknotting sequence by $K=K_0,K_1,\ldots,K_{m-1},K_m=K'$, where $m=u(K,K')$ and each consecutive pair of knots differs by a crossing change.
For every $i=0,\ldots,m$, we can represent $K_i$ by a grid $\mathbb{G}_i$ such that every grid has the same size, and for $i=0,\ldots,m-1$, the grids $\mathbb{G}_i$ and $\mathbb{G}_{i+1}$ differ by a cross-commutation of columns (cf. \cite{grid-homology}, Section 6.2).
By Theorem~\ref{thm:filt_6.1.1}, there exist maps
\[
    \mathfrak{f}_i^+ : \mathcal{GC}^-(\mathbb{G}_i) \to \mathcal{GC}^-(\mathbb{G}_{i+1}) \;\; \text{and} \;\; \mathfrak{f}_i^- : \mathcal{GC}^-(\mathbb{G}_{i+1}) \to \mathcal{GC}^-(\mathbb{G}_{i})
\]
such that $\mathfrak{f}_i^+ \circ \mathfrak{f}_i^-$ and $\mathfrak{f}_i^- \circ \mathfrak{f}_i^+$ are filtered chain-homotopic to multiplication by $V_1$.
So, there exist maps
\[
    \mathfrak{f}^+ : \mathcal{GC}^-(\mathbb{G}_0) \to \mathcal{GC}^-(\mathbb{G}_{m}) \;\; \text{and} \;\; \mathfrak{f}^- : \mathcal{GC}^-(\mathbb{G}_{m}) \to \mathcal{GC}^-(\mathbb{G}_{0})
\]
such that $\mathfrak{f}_i^+ \circ \mathfrak{f}_i^-$ and $\mathfrak{f}^- \circ \mathfrak{f}^+$ are filtered chain-homotopic to multiplication by $V_1^m$.

We make the following definition.
\begin{definition}\label{def:l}
Let $K$ and $K'$ be two knots, and let $\mathbb{G}$ and $\mathbb{G}'$ be grids of the same size representing $K$ and $K'$ respectively. 
Consider all pairs of maps
\[
    \mathfrak{f}^+ : \mathcal{GC}^-(\mathbb{G}) \to \mathcal{GC}^-(\mathbb{G}') \;\; \text{and} \;\; \mathfrak{f}^- : \mathcal{GC}^-(\mathbb{G}') \to \mathcal{GC}^-(\mathbb{G})
\]
such that $\mathfrak{f}^+$ and $\mathfrak{f}^-$ are homogeneous filtered chain maps,
and $\mathfrak{f}^+ \circ \mathfrak{f}^-$ and $\mathfrak{f}^- \circ \mathfrak{f}^+$ are filtered chain-homotopic to multiplication by $V_1^m$ for some positive integer $m$.
Let $\mathfrak{l}(K,K')$ the minimal $m$ over all pairs $\mathfrak{f}^+,\mathfrak{f}^-$ satisfying the above conditions.
Let $\mathfrak{l}(K) = \mathfrak{l}(K,U)$, where $U$ is the unknot.
\end{definition}

By definition, the unknotting number $u(K) = u(K,U)$.
The above discussion implies
\begin{theorem}\label{thm:unknotting}
$\mathfrak{l}(K,K') \le u(K,K')$ and $\mathfrak{l}(K) \le u(K)$.
\end{theorem}

\begin{remark}
The invariant $\mathfrak{l}(K)$ can be computed for knots such that the chain homotopy type of the filtered grid complex is known, such as alternating knots.
Compare Section 14.2 of \cite{grid-homology}.
\end{remark}

\section{Comparison with Alishahi--Eftekhary knot invariant}\label{sec:AE}
Alishahi and Eftekhary define an invariant $\mathfrak{l}_{AE}(K)$,
which is a lower bound on the unknotting number $u(K)$, and an upper bound on the concordance invariant $\nu^+(K)$ and also an upper bound on $\widehat{\mathfrak{t}}(K)$, where $\widehat{\mathfrak{t}}(K)$ is the maximum order of $U$-torsion in knot Floer homology $\mathrm{HFK}^-(K)$, compare Definition 3.1 of \cite{AE20}.
We copy the definition of the Aliashahi--Eftekhary invariant in Definition~\ref{def:lAE}.
For further properties of $\mathfrak{l}(K)$, see Theorem 1.1 and Corollary 1.2 of \cite{AE20}.
The goal of this section is to show the two definitions of $\mathfrak{l}$ coincide: for any two knot $K \subset S^3$, $\mathfrak{l}(K)=\mathfrak{l}_{AE}(K)$.

The definition of $\mathfrak{l}_{AE}(K)$ is in terms of a knot Floer complex $\mathrm{CF}(K)$, which is obtained from a sutured manifold in their construction \cite{AE-suture-floer}, a refinement of Juh\'asz's construction \cite{juhasz}. 
The knot chain complex $\mathrm{CF}(K)$ is a module over $\mathbb{F}[u,w]$ is $\ZZ$--bigraded, with Maslov and Alexander gradings as defined in \cite{OS04}.
The complex $\CF(K)$ is the chain homotopy type the complex $\CF(\mathcal{H})$, where $\mathcal{H}$ is a Heegaard diagram of the sutured manifold associated to the knot. 
The sutured manifold associated to the knot is the complement of a neighborhood of a knot in $S^3$, and it has two sutures on the boundary torus oriented in opposite directions.

If a knot $K' \subset S^3$ is obtained from the knot $K \subset S^3$ by a crossing change, Alishahi--Eftekhary prove, that there exist homogeneous chain maps
\[
    \mathfrak{f}^- : \mathrm{CF}(K) \to \mathrm{CF}(K') \;\; \text{and} \;\;
    \mathfrak{f}^+ : \mathrm{CF}(K') \to \mathrm{CF}(K)
\]
which preserve the Maslov grading, such that $\mathfrak{f}^+ \circ \mathfrak{f}^-$ and $\mathfrak{f}^- \circ \mathfrak{f}^+$ are chain homotopic to multiplication by $w$, see Theorem 2.3 of \cite{AE20}.

We define the Alishahi--Eftekhary knot invariant below.

\begin{definition}\label{def:lAE}
Given two knots $K,K'$, consider all pairs of homogeneous maps
\[
    \mathfrak{f}^- : \mathrm{CF}(K) \to \mathrm{CF}(K') \;\; \text{and} \;\;
    \mathfrak{f}^+ : \mathrm{CF}(K') \to \mathrm{CF}(K)
\]
which are homogeneous, preserve the Maslov grading, and $\mathfrak{f}^+ \circ \mathfrak{f}^-$ and $\mathfrak{f}^- \circ \mathfrak{f}^+$ are chain homotopic to multiplication by $w^m$.
Then $\mathfrak{l}_{AE}(K,K')$ is the minimum nonnegative integer $m$ such that there exist maps $\mathfrak{f}^-,\mathfrak{f}^+$ satisfying the previous conditions.
If $U$ is the unknot, let $\mathfrak{l}_{AE}(K) = \mathfrak{l}_{AE}(K,U)$.
\end{definition}

Our goal in this section is to show our invariant $\mathfrak{l}(K)$ from Definition~\ref{def:l} corresponds to the Alishahi--Eftekhary $\mathfrak{l}$-invariant:

\begin{theorem}~\label{thm:l=lAE}
For any two knots $K,K'$, $\mathfrak{l}(K,K') = \mathfrak{l}_{AE}(K,K')$.
In particular, $\mathfrak{l}(K)=\mathfrak{l}_{AE}(K)$.
\end{theorem}

\begin{proof}
Fix two diagrams $\GG$ and $\GG'$ of size $n$ for $K$ and $K'$ respectively.
First we show that $\mathfrak{l}(K,K')$ only depends on the filtered chain homotopy type of $\mathcal{GC}^-(\GG)$ and $\mathcal{GC}^-(\mathbb{G}')$.
Fix two maps
\[
    \mathfrak{f}^+ : \mathcal{GC}^-(\mathbb{G}) \to \mathcal{GC}^-(\mathbb{G}') \;\; \text{and} \;\; \mathfrak{f}^- : \mathcal{GC}^-(\mathbb{G}') \to \mathcal{GC}^-(\mathbb{G})
\]
such that for $m = \mathfrak{l}(K,K')$, we have $\mathfrak{f}^+ \circ \mathfrak{f}^- \simeq V_1^m$ and $\mathfrak{f}^- \circ \mathfrak{f}^+ \simeq V_1^m$, where $\simeq$ is filtered chain homotopy equivalence.
Let $\mathcal{C}$ be a filtered chain complex in the filtered quasi-isomorphism class of $\mathcal{GC}^-(\mathbb{G'})$. 
Since $\mathcal{GC}^-(\mathbb{G})$ is a chain complex freely generated over a polynomial ring, filtered quasi-isomorphism and filtered chain homotopy are equivalent relations on filtered chain complexes, see for example Appendix A.8 of \cite{grid-homology}.
So there are maps $\varphi : \mathcal{GC}^-(\mathbb{G'}) \to \mathcal{C}$ and $\psi : \mathcal{C} \to \GC^-(\mathbb{G})$ such that $\varphi \circ \psi \simeq \id_{\mathcal{C}}$ and $\psi \circ \varphi \simeq \id_{\GC^-(\GG)}$.
In summary, we have the diagram
\begin{center}
\begin{tikzcd}
\GC^-(\GG) \arrow[bend left]{r}{\mathfrak{f}^+} & \GC^-(\GG') \arrow[bend left]{l}{\mathfrak{f}^-} \arrow[bend left]{r}{\varphi} & \mathcal{C} \arrow[bend left]{l}{\psi}.
\end{tikzcd}
\end{center}
Consider $\mathfrak{g}^+ = \varphi \circ \mathfrak{f}^+$ and $\mathfrak{g}^- = \mathfrak{f}^- \circ \psi$. 
Then $\mathfrak{g}^+ \circ \mathfrak{g}^- = (\varphi \circ \mathfrak{f}^+) \circ (\mathfrak{f}^- \circ \psi) \simeq \varphi \circ V_1^m \circ \psi = V_1^m \circ \varphi \circ \psi \simeq V_1^m$ and $\mathfrak{g}^- \circ \mathfrak{g}^- = (\mathfrak{f}^- \circ \psi) \circ (\varphi \circ \mathfrak{f}^+) \simeq \mathfrak{f}^+ \circ \mathfrak{f}^- \simeq V_1^m$.
Similarly, replacing $\mathcal{GC}^-(\GG)$ with any filtered chain complex in its filtered quasi-isomorphism class does not change $m$.
Therefore, $\mathfrak{l}(K,K')$ only depends on the filtered chain homotopy type of $\mathcal{GC}^-(\GG)$ and $\mathcal{GC}^-(\mathbb{G}')$.

Now we show that the filtered chain homotopy class $\GC^-(K)$ agrees with the chain homotopy class $\CF(K)$.
We have an isomorphism of filtered chain complexes $\GC^-(\GG) \cong \CFK^{-,*}(\mathcal{H})$, where $\mathcal{H}$ is the Heegaard diagram corresponding to $\GG$, and both complexes are over the polynomial ring $\FF[V_1,\ldots,V_n]$, see Theorem 16.4.1 in \cite{grid-homology}.
By Proposition 2.3 of \cite{manolescu-ozsvath-sarkar}, the filtered chain homotopy type of the $2n$-basepoint knot Floer complex $\CFK^{-,*}(\mathcal{H})$ agrees with the filtered chain homotopy type of the standard knot Floer complex $\CFK^-(K)$, a filtered module over $\FF[U]$.
If we instead interpret $\CFK^-(K)$ as a bigraded complex over the two-variable ring $\FF[u,v]$ by letting the second variable be the filtration level, we get the knot Floer complex $\CF(K)$.

By definition $\mathfrak{l}_{AE}(K,K')$ depends only on $\CF(K)$ and $\CF(K')$, and as we showed above, $\mathfrak{l}(K,K')$ depends only on the filtered chain homotopy classes $\GC^-(K)$ and $\GC^-(K')$.
The correspondence of filtered chain homotopy class $\GC^-(\GG)$ and chain homotopy class $\CF(K)$ implies $\mathfrak{l}(K,K')= \mathfrak{l}_{AE}(K,K')$.
\end{proof}

\bibliography{bibliography}{}
\bibliographystyle{alpha}

\end{document}